\documentclass[draft,12pt]{amsart}

\usepackage{fullpage}
\usepackage{url}
\usepackage{epic}
\usepackage{eepic}
\newtheorem{theorem}{Theorem}[section]
\theoremstyle{plain}

\newtheorem{corollary}[theorem]{Corollary}

\newtheorem{example}[theorem]{Example}

\newtheorem{lemma}[theorem]{Lemma}

\newtheorem{problem}[theorem]{Problem}
\newtheorem{proposition}[theorem]{Proposition}

\numberwithin{equation}{section}

\newcommand{\il}[1]{\mathcal{C}_{\lambda}(#1)}
\newcommand{\ir}[1]{\mathcal{C}_{\rho}(#1)}
\newcommand{\rmlt}{\mathrm{Mlt}_{\rho}}
\newcommand{\lmlt}{\mathrm{Mlt}_{\lambda}}
\newcommand{\linn}{\mathrm{Inn}_{\lambda}}
\newcommand{\aut}{\mathrm{Aut}}

\begin{document}

\title{Incidence properties of cosets in loops}

\author{Michael Kinyon}
\email[Kinyon]{mkinyon@math.du.edu}

\author{Kyle Pula}
\email[Pula]{jpula@du.edu}

\author{Petr Vojt\v{e}chovsk\'y}
\email[Vojt\v{e}chovsk\'y]{petr@math.du.edu}

\address{Department of Mathematics, University of Denver, 2360 S. Gaylord St, Denver, Colorado 80208, U.S.A.}

\begin{abstract}
We study incidence properties among cosets of finite loops, with emphasis on well-structured varieties such as antiautomorphic loops and Bol loops. While cosets in groups are either disjoint or identical, we find that the incidence structure in general loops can be much richer. Every symmetric design, for example, can be realized as a canonical collection of cosets of a finite loop. We show that in the variety of antiautomorphic loops the poset formed by set inclusion among intersections of left cosets is isomorphic to that formed by right cosets. We present an algorithm that, given a finite Bol loop $S$, can in some cases determine whether $|S|$ divides $|Q|$ for all finite Bol loops $Q$ with $S \le Q$, and even whether there is a selection of left cosets of $S$ that partitions $Q$. This method results in a positive confirmation of Lagrange's Theorem for Bol loops for a few new cases of subloops. Finally, we show that in a left automorphic Moufang loop $Q$ (in particular, in a commutative Moufang loop $Q$), two left cosets of $S\le Q$ are either disjoint or they intersect in a set whose cardinality equals that of some subloop of $S$.
\end{abstract}

\keywords{Cosets in loops, incidence properties of cosets, coset partition, combinatorial design, Lagrange's Theorem, Bol loop, Moufang loop}

\subjclass[2010]{20N05, 05B05}

\thanks{All three authors supported by the 2008 PROF grant of the University of Denver.}

\maketitle

\section{Introduction}\label{sec:intro}

This paper is intended for both design theorists and loop theorists. In order to make it as self-contained as possible, we therefore present basic definitions and results from both fields. The interested reader can find this necessary background material and much more in \cite{BeJuHa}, \cite{Br}, \cite{HuPi} and \cite{Pf}.

A \emph{quasigroup} is a groupoid $(Q,\cdot)$ such that for every $a$, $b\in Q$ the equations $ax=b$, $ya=b$ have unique solutions $x$, $y\in Q$, respectively. A \emph{loop} is a quasigroup $(Q,\cdot)$ with \emph{neutral element} $1\in Q$ satisfying $1x = x1 = x$ for every $x\in Q$. A nonempty subset $S\subseteq (Q,\cdot)$ is a \emph{subloop} of $Q$, which we denote by $S\le Q$, if $(S,\cdot)$ is a loop in its own right.

For a loop $Q$, subloop $S\le Q$, and $x\in Q$, the \emph{left} (resp. \emph{right}) \emph{coset} of $S$ with representative $x$ is the set $xS = \{xs:\;s\in S\}$ (resp. $Sx = \{sx:\;s\in S\}$). Cosets play a central role in proofs of some of the most basic results in the theory of groups, such as Lagrange's Theorem that $|S|$ divides $|Q|$, which is obtained by showing that the left (and right) cosets of $S$ form a partition of $Q$.

In contrast to such elegant yet boring incidence properties of cosets in the associative case, the incidence properties of cosets in nonassociative loops are very rich but very poorly understood. In this paper, we take up the study of coset incidence in nonassociative loops, emphasizing several of the well-structured varieties such as antiautomorphic loops and Bol loops. Our results are rather incomplete, and the paper should be viewed as a point of departure for a more systematic study.

Our motivation is twofold. First, we would like to find an elementary proof of Lagrange's Theorem for Moufang loops. Recall that \emph{Moufang loops} are defined by any one of the four equivalent identities
\begin{equation}\label{Eq:Moufang}
    ((xy)x)z{=}x(y(xz)),\,((xy)z)y{=}x(y(zy)),\,(xy)(zx){=}x((yz)x),\,(xy)(zx){=}(x(yz))x,
    \end{equation}
and are probably the most studied variety of nonassociative loops. Two groups of authors \cite{GaHa}, \cite{GrZa} independently proved Lagrange's Theorem for Moufang loops. However, their proofs rely on Liebeck's classification of finite simple Moufang loops \cite{Li}, which in turn depends on the classification of finite simple groups!

Secondly, we are intrigued by the possibility of realizing interesting combinatorial designs as cosets in algebraically structured loops. Let us illustrate this idea by two examples:

Consider the loop $(Q,\cdot)$ with the following multiplication table
\begin{displaymath}
\begin{small}
\begin{array}{c|cccccccccc}
Q&0&1&2&3&4&5&6&7&8&9\\
 \hline
0&0&1&2&3&4&5&6&7&8&9\\
1&1&2&0&4&5&6&7&8&9&3\\
2&2&0&1&6&7&8&9&3&4&5\\
3&3&4&6&2&0&7&5&9&1&8\\
4&4&5&7&0&3&9&8&1&6&2\\
5&5&6&8&7&9&3&1&4&2&0\\
6&6&7&9&5&8&1&3&2&0&4\\
7&7&8&3&9&1&4&2&0&5&6\\
8&8&9&4&1&6&2&0&5&3&7\\
9&9&3&5&8&2&0&4&6&7&1
\end{array}
\end{small}
\end{displaymath}
and subloop $S = \{0,1,2\} \le Q$. It is easy to check that both $\{xS:\;x\in Q\setminus S\}$ and $\{Sx:\;x\in Q\setminus S\}$ are isomorphic as designs to the projective plane of order $2$. In \S \ref{sec:designs}, we observe that in fact every symmetric design can be realized in an analogous way. A particularly interesting aspect of this example, however, is that $(Q, \cdot)$ possesses some algebraic structure. It happens to be a commutative weak inverse property loop, a representative of one of four isomorphism classes with these properties that realize the projective plane of order $2$ in this way.
%(This example happens to be from the unique such class that minimizes the number of elements that are squares.)

As a second illustration, consider the smallest nonassociative Moufang loop \cite{Chein}, the loop $(M,\cdot)$ with multiplication table
\begin{displaymath}
    \begin{small}
    \begin{array}{c|cccccccccccc}
        M&1&2&3&4&5&6&7&8&9&a&b&c\\
        \hline
        1&1&2&3&4&5&6&7&8&9&a&b&c\\
        2&2&1&4&3&6&5&8&7&c&b&a&9\\
        3&3&6&5&2&1&4&9&a&b&c&7&8\\
        4&4&5&6&1&2&3&a&9&8&7&c&b\\
        5&5&4&1&6&3&2&b&c&7&8&9&a\\
        6&6&3&2&5&4&1&c&b&a&9&8&7\\
        7&7&8&b&a&9&c&1&2&5&4&3&6\\
        8&8&7&c&9&a&b&2&1&4&5&6&3\\
        9&9&c&7&8&b&a&3&4&1&6&5&2\\
        a&a&b&8&7&c&9&4&3&6&1&2&5\\
        b&b&a&9&c&7&8&5&6&3&2&1&4\\
        c&c&9&a&b&8&7&6&5&2&3&4&1
    \end{array}
    \end{small}
\end{displaymath}
Note that $(2\cdot 3)\cdot 7 \ne 2 \cdot (3\cdot 7)$, and that $S=\{1,2,7,8\} \le M$. Precisely four of the left cosets of $S$ are necessarily equal to $S$, while the remaining $8$ left cosets are as follows:
\begin{displaymath}
    \begin{array}{llll}
    B_3 = \{3,6,9,a\},&B_4 = \{4,5,9,a\},&B_5 = \{4,5,b,c\},&B_6 = \{3,6,b,c\}\\
    B_9 = \{3,4,9,c\},&B_{a} = \{3,4,a,b\},&B_{b} = \{5,6,a,b\},&B_{c} = \{5,6,9,c\}.
    \end{array}
\end{displaymath}
Let $\mathcal P = \{3,4,5,6,9,a,b,c\}$ and $\mathcal B =\{B_3,B_4,B_5,B_6,B_9,B_{a},B_{b},B_{c}\}$. Let $\mathcal D = (\mathcal P,\mathcal B)$ be the corresponding incidence structure. It is easy to see that $\mathcal D$ is a $1$-$(8,4,4)$ design, so every point is contained in precisely $4$ blocks. More importantly, $\mathcal D$ is close to being a $2$-design. Indeed, any two points of $\mathcal P$ are contained in precisely two blocks, except for the pairs of points $\{3,5\}$, $\{4,6\}$, $\{9,b\}$ and $\{a,c\}$, none of which is contained in any block.

Of course, there is no $2$-$(8,4,2)$ design, so our effort was doomed from the start, but it strikes us as a rather elegant way of obtaining a near $2$-design.

\section{Preliminaries}

Throughout this paper, let $Q$ be a finite loop of order $n$ and $S$ a subloop of $Q$ of order $m$. We study the incidence properties of the sets
\begin{displaymath}
    \il{Q,S} = \left\{\bigcap_{x\in X}xS\,:\,   \emptyset\ne X\subseteq Q\right\},\quad
    \ir{Q,S} = \left\{\bigcap_{x\in X}Sx\,:\,   \emptyset\ne X\subseteq Q\right\},
\end{displaymath}
partially ordered by inclusion. We are particularly interested in the maximal elements of $\il{Q,S}$, say, namely the left cosets $xS$, $x\in Q$.

Let us first address the question of which combinations of $n$, $m$ are possible. The answer follows easily from the following stronger result \cite[Theorem 2]{Ryser}:

\begin{theorem}[Ryser] \label{Th:Ryser} Let $R$ be an $r\times s$ array containing symbols from the set $\{1,\dots,n\}$. Suppose that every symbol $1\le i\le n$ occurs at most once in every column of $R$ and at most once in every row of $R$. For $1\le i\le n$, let $\ell(i)$ be the number of occurrences of $i$ in $R$. Then $R$ can be embedded into a latin square of order $n$ if and only if
\begin{equation}\label{Eq:Ryser}
    \ell(i)\ge r+s-n
\end{equation}
holds for every $1\le i\le n$.
\end{theorem}

We now easily derive the desired restriction on $n$ and $m$ (see \cite[Theorem 1.5.1]{DeKe} or \cite[Theorem 2]{Evans}):

\begin{corollary}\label{Cr:PossibleOrders}
Let $1\le m\le n$. Then there is a loop $Q$ of order $n$ with a subloop $S$ of order $m$ if and only if either $m=n$ or $m\le \lfloor n/2\rfloor$.
\end{corollary}
\begin{proof}
The case $m=n$ is obvious, so assume that $m<n$. If $S$ is a subloop of order $m$ in a loop of order $n$ then any multiplication table $R$ of $S$ is a latin square of order $m$, without loss of generality containing the symbols $\{1,\dots,m\}\subseteq\{1,\dots,n\}$. Given $1\le i\le n$ and letting $\ell(i)$ be as in Theorem \ref{Th:Ryser}, we have
\begin{displaymath}
    \ell(i) = \left\{\begin{array}{ll}
        m,&\text{ if }1\le i\le m,\\
        0,&\text{ if }m<i\le n.
    \end{array}\right.
\end{displaymath}
We have $\ell(i)=0$ for some $i$ (since $m<n$), so \eqref{Eq:Ryser} holds for every $1\le i\le n$ if and only if $0\ge m+m-n$, i.e., $m\le \lfloor n/2\rfloor$. By Theorem \ref{Th:Ryser}, $R$ embeds into a latin square $L$ of order $n$ if and only if $m\le \lfloor n/2\rfloor$. Upon permuting the rows and columns of $L$ as needed, we can consider $L$ to be a multiplication table of a loop $Q$ of order $n$.
\end{proof}

\section{The left-right symmetry}\label{Sc:Symmetries}

In the terminology of partially ordered sets, both $\il{Q,S}$ and $\ir{Q,S}$ are \emph{meet-semilat\-tices} (for every $a$, $b$ there exists a largest lower bound $a\wedge b$) in which the maximal elements are \emph{meet-dense} (every element can be expressed as a finite meet of maximal elements). Note that we do not require that meet-semilattices have a largest element.

A bijection $f:A\to B$ between two meet-semilattices is an \emph{isomorphism} if $f(a\wedge b) = f(a)\wedge f(b)$ for every $a$, $b\in A$. Such an isomorphism clearly maps maximal elements of $A$ to maximal elements of $B$.

The following example shows that $\il{Q,H}$ and $\ir{Q,H}$ need not be isomorphic:

\begin{example}
Consider the loop $Q$ with multiplication table
\begin{displaymath}
\begin{array}{c|ccccccc}
    &1&2&3&4&5&6\\
    \hline
    1&1&2&3&4&5&6\\
    2&2&1&4&3&6&5\\
    3&3&4&5&6&1&2\\
    4&4&5&6&1&2&3\\
    5&5&6&1&2&3&4\\
    6&6&3&2&5&4&1
\end{array}
\end{displaymath}
and subloop $H=\{1,2\}$. Then there are five left cosets $\{1,2\}$, $\{3,4\}$, $\{3,6\}$, $\{4,5\}$, $\{5,6\}$ but only three right cosets $\{1,2\}$, $\{3,4\}$, $\{5,6\}$. This is a smallest loop in which the number of left cosets (with respect to a fixed subloop) does not coincide with the number of right cosets.
\end{example}

Moreover, in an arbitrary meet-semilattice, if the maximal elements are meet-dense, then the isomorphism is determined by its values on the maximal elements. Not every bijection of maximal elements can be extended into an isomorphism, of course, but the following result gives the necessary and sufficient condition:

\begin{lemma}\label{Lm:ExtendMax}
Let $P$, $P'$ be meet-semilattices in which maximal elements are meet-dense. Let $M=\{m_i\,:\,   i\in I\}$, $M'$ be the sets of all maximal elements of $P$, $P'$, respectively, and let $f:M\to M'$ be a bijection. Then $f$ extends into an isomorphism $P\to P'$ if and only if for every $\emptyset\ne J$, $K\subseteq I$ we have
\begin{equation}\label{Eq:Extend}
    \bigwedge_{j\in J}m_j = \bigwedge_{k\in K}m_k \Leftrightarrow \bigwedge_{j\in J}f(m_j) = \bigwedge_{k\in K}f(m_k).
\end{equation}
\end{lemma}
\begin{proof}
Assume that \eqref{Eq:Extend} holds. Since the maximal elements of $P$ are meet-dense, every element of $P$ can be expressed as $\bigwedge_{j\in J}m_j$ for some $\emptyset\ne\ J\subseteq I$. If $f$ is to be a homomorphism, we must set $f(\bigwedge_{j\in J}m_j) = \bigwedge_{j\in J}f(m_j)$. By the direct implication of \eqref{Eq:Extend}, $f$ is well-defined. With $a=\bigwedge_{j\in J}m_j$, $b=\bigwedge_{k\in K}m_k$, we have $f(a\wedge b) = f(\bigwedge_{\ell\in J\cup K}m_\ell) = \bigwedge_{\ell\in J\cup K}f(m_\ell) = \bigwedge_{j\in J}f(m_j)\wedge\bigwedge_{k\in K}f(m_k) = f(a)\wedge f(b)$ because $P$, $P'$ are meet-semilattices. Thus $f$ is a homomorphism, and it is one-to-one thanks to the indirect implication of \eqref{Eq:Extend}. Given $a'\in P'$, we have $a'=\bigwedge_{j\in J}f(m_j)$ for some $\emptyset\ne J\subseteq I$ since the maximal elements in $P'$ are meet-dense, and thus $f(\bigwedge_{j\in J}m_j) = \bigwedge_{j\in J}f(m_j) = a'$, proving that $f$ is onto $P'$.

Conversely, if $f$ extends into an isomorphism, we must have $f(\bigwedge_{j\in J}m_j) = \bigwedge_{j\in J}f(m_j)$ for every $\emptyset\ne J\subseteq I$. If $\bigwedge_{j\in J}m_j = \bigwedge_{k\in K}m_k$ then $f(\bigwedge_{j\in J}m_j) = f(\bigwedge_{k\in K}m_k)$, and so $\bigwedge_{j\in J}f(m_j) = \bigwedge_{k\in K}f(m_k)$. The converse is also true, since $f$ is one-to-one.
\end{proof}

In some situations an isomorphism between $\il{Q,S}$ and $\ir{Q,S}$ can be deduced without constructing it explicitly. For example, if $Q$ is a group and $S\le Q$ then any two left (right) cosets of $S$ either coincide or are disjoint, hence $\il{Q,S}$ and $\ir{Q,S}$ are isomorphic. Similarly, if $Q$ is a loop and $S$ is a \emph{normal subloop} of $Q$ (that is, $xS=Sx$, $x(yS) = (xy)S$, $x(Sy) = (xS)y$ for every $x$, $y\in Q$) then again any two left (right) cosets of $S$ either coincide or are disjoint, so $\il{Q,S}$ and $\ir{Q,S}$ are isomorphic. Finally, note that in commutative loops the isomorphism holds trivially.

Let us nevertheless construct an explicit isomorphism $f:\il{Q,S}\to\ir{Q,S}$ when $Q$ is a group. The first candidate $f(xS) = Sx$ fails to do the job because it is not necessarily well-defined; it is possible to have $xS=yS$ but $Sx\ne Sy$, a smallest counterexample being the symmetric group $Q=S_3=\{\sigma,\,\rho\,:\,   \sigma^2=\rho^3=(\sigma\rho)^2=1\}$ with subgroup $S=\langle \sigma\rho\rangle$. But the next idea $f(xS) = Sx^{-1}$ works for groups and can be generalized:

A loop $Q$ has the \emph{antiautomorphic inverse property} (\emph{AAIP}) if for every $x\in Q$ there is $x^{-1}\in Q$ such that $xx^{-1}=1=x^{-1}x$ and if $(xy)^{-1} = y^{-1}x^{-1}$ holds for every $x$, $y\in Q$.

\begin{proposition}\label{Pr:AAIP}
Let $Q$ be a loop with the antiautomorphic inverse property and let $S\le Q$. Then the mapping $f:xS\mapsto Sx^{-1}$ is well-defined, and extends uniquely into an isomorphism $\il{Q,S}\to\ir{Q,S}$.
\end{proposition}
\begin{proof}
Consider the antiautomorphism $f:Q\to Q$, $x\mapsto x^{-1}$. For any subset $X$ of $Q$, let $f(X) = \{f(x)\,:\,   x\in X\}$. In particular, $f(xS) = (xS)^{-1} = S^{-1}x^{-1} = Sx^{-1}$. Note that $f$ is a homomorphism $\il{Q,S}\to\ir{Q,S}$, as $f(xS\cap yS) = (xS\cap yS)^{-1} = Sx^{-1}\cap Sy^{-1} = f(xS)\cap f(yS)$. The direct implication in \eqref{Eq:Extend} is therefore satisfied. The indirect implication holds as well, since $f^{-1}=f$.
\end{proof}

The variety of loops with the AAIP contains many well-studied varieties of loops. For instance, \emph{diassociative loops} (any two elements generate a group), \emph{inverse property loops} (satisfying $x^{-1}(xy) = y$ and $(xy)y^{-1}=x$), the already-mentioned Moufang loops, and \emph{automorphic loops} (inner mappings are automorphisms; see \cite{BrPa} and \cite{JeKiVo}).

Indeed, inverse property loops have the AAIP since $(xy)^{-1}x = (xy)^{-1}((xy)y^{-1}) = y^{-1}$, so $(xy)^{-1} = ((xy)^{-1}x)x^{-1} = y^{-1}x^{-1}$. Moufang loops have the AAIP because they are inverse property loops. In fact, Moufang loops are diassociative, by the famous Moufang's Theorem. (See \cite{Moufang} for the original proof of Moufang's Theorem and \cite{Drapal} for a much shorter proof.) Automorphic loops have the AAIP by \cite[Theorem 7.5]{JoKiNaVo}.

\begin{problem}
Is there a variety (or class) $\mathcal V$ of loops not contained in the varieties of antiautomorphic inverse property loops or commutative loops such that for every finite $Q\in\mathcal V$ and every $S\le Q$ the two meet-semilattices $\il{Q,S}$, $\ir{Q,S}$ are isomorphic?
\end{problem}

Given a loop $Q$, denote by $Q^{\mathrm{op}} = (Q,*)$ the loop with operation $x*y = yx$. Then, clearly, $\il{Q,S}\cong \ir{Q^{\mathrm{op}},S}$ because the two sets are in fact equal. We can therefore restrict our attention to $\il{Q,S}$ from now on.

\section{Symmetric designs and cosets}\label{sec:designs}

Figure \ref{Fg:Cayley} depicts a multiplication table of a loop $Q$ with a subloop $S$. The $m\times m$ latin square $L_1$ contains only elements of $S$ and is a multiplication table of $S$. The $(n-m)\times m$ \emph{latin rectangle} (that is, no symbol is repeated in any row or column) $L_2$ contains only symbols of $Q\setminus S$, each symbol of $Q\setminus S$ occurs in every column of $L_2$ precisely once, and each symbol of $Q\setminus S$ occurs in precisely $m$ rows of $L_2$. The $n\times (n-m)$ latin rectangle $L_3$ completes the multiplication table of $Q$.

\begin{figure}[ht]
\begin{center}
\setlength{\unitlength}{0.254mm}
\begin{picture}(161,151)(10,-165)
        \allinethickness{0.254mm}\path(20,-30)(155,-30)(155,-165)(20,-165)(20,-30) % Plain Solid Square
        \allinethickness{0.254mm}\path(20,-30)(75,-30)(75,-85)(20,-85)(20,-30) % Plain Solid Square
        \allinethickness{0.254mm}\path(75,-85)(75,-165) % Plain Solid Line
        \put(40,-64){\shortstack{$L_1$}} % Plain Text
        \put(40,-130){\shortstack{$L_2$}} % Plain Text
        \put(106,-104){\shortstack{$L_3$}} % Plain Text
        \put(44,-22){\shortstack{$S$}} % Plain Text
        \put(2,-64){\shortstack{$S$}} % Plain Text
        \put(-27,-130){\shortstack{$Q\setminus S$}} % Plain Text
        \put(95,-22){\shortstack{$Q\setminus S$}} % Plain Text
         % Set color to black again (default font color)
\end{picture}
\end{center}
\caption{Nested multiplication tables of a loop $Q$ and its subloop $S$.}\label{Fg:Cayley}
\end{figure}
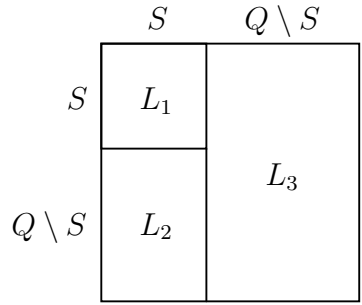

In this section we are concerned with design-like properties of the left cosets $\{xS\,:\,   x\in Q\}$, i.e., the maximal elements of $\il{Q,S}$. For convenience, we count identical cosets with the appropriate multiplicity for a total of $n$ left cosets. As $S$ is a subloop of $Q$, we have $xS=S$ for every $x\in S$ and $yS\cap S=\emptyset$ for every $y\in Q\setminus S$. Interesting incidence properties can therefore be found only among the $n-m$ cosets
\begin{displaymath}
    \mathcal B(Q,S) = \{xS\,:\,   x\in Q\setminus S\}.
\end{displaymath}
Note that the cosets of $\mathcal B(Q,S)$ correspond to the rows of $L_2$ in Figure \ref{Fg:Cayley}.

Let $\mathcal P$ be a set of points and $\mathcal B$ a collection of subsets of $\mathcal P$, called \emph{blocks}. Then $\mathcal D = (\mathcal P,\mathcal B)$ is a $t$-$(v,k,\lambda)$ \emph{(balanced incomplete block) design} if $|\mathcal P|=v$, $|B|=k$ for every $B\in\mathcal B$, and if every $t$-element subset of $\mathcal P$ is contained in precisely $\lambda>0$ blocks of $\mathcal B$. While it is sometimes assumed that $t\ge 2$ in the definition of a design, we allow $t=1$, too.

It can be easily shown by double counting that a $t$-design is also a $t'$-design for all $1\le t'\le t$. In particular, if $\mathcal D=(\mathcal P,\mathcal B)$ is a $t$-design, there is a constant $r$ such that every point of $\mathcal P$ is contained in precisely $r$ blocks of $\mathcal B$. The design $\mathcal D$ is called \emph{symmetric} if $b=|\mathcal B|$ is equal to $v=|\mathcal P|$. Equivalently, $\mathcal D$ is symmetric if $r=k$. (By elementary arguments, a symmetric design with $k<v-1$ must have $t\le 2$.)

As mentioned above, every element of $Q\setminus S$ is contained in precisely $m$ rows of $L_2$. Hence
\begin{displaymath}
    \mathcal D(Q,S) = (Q\setminus S,\ \mathcal B(Q,S)) = (Q\setminus S,\ \{xS\,:\,   x\in Q\setminus S\})
\end{displaymath}
is at least a $1$-$(n-m,m,m)$ design, possibly a $t$-design with $t>1$. Our immediate goal is to prove that all symmetric designs can be realized by cosets in loops:

\begin{theorem}\label{Th:Designs}
Let $Q$ be a loop and let $S$ be a subloop of $Q$. Then $\mathcal D(Q,S)$ is a symmetric $1$-design. Conversely, if $\mathcal D$ is a symmetric design, then there is a loop $Q$ and a subloop $S\le Q$ such that $\mathcal D = \mathcal D(Q,S)$.
\end{theorem}

The key steps in the proof of Theorem \ref{Th:Designs} are furnished by Lemma \ref{Lm:Marriage} (a well-known result) and Lemma \ref{Lm:HallCompletion} (a special case of Theorem \ref{Th:Ryser}).

Given a family $\mathcal B$ of subsets of $\mathcal P$, we say that $g:\mathcal B\to \mathcal P$ is a \emph{system of distinct representatives} if $g(B)\in B$ for every $B\in\mathcal B$ and $g$ is one-to-one. Two systems of distinct representatives are said be \emph{disjoint} if they disagree on every $B \in \mathcal B$.

\begin{lemma}\label{Lm:Marriage}
Let $\mathcal B = \{B_1,\dots,B_n\}$ be a family of $k$-element subsets of $\mathcal P = \{1,\dots, n\}$. The following statements are equivalent:
\begin{enumerate}
\item[(i)] Each element of $\mathcal P$ lies in precisely $k$ blocks of $\mathcal B$ (i.e., $(\mathcal P,\mathcal B)$ is a $1$-$(n,k,k)$ design).
\item[(ii)] The family $\mathcal B$ has $k$ mutually disjoint systems of distinct representatives.
\item[(iii)] It is possible to form an $n\times k$ latin rectangle $L$ so that the symbols in the $i$th row of $L$ are the elements of the block $B_i$.
\end{enumerate}
\end{lemma}
\begin{proof}
Suppose we are given (i). Any collection of $s$ blocks from $\mathcal B$ contains $ks$ points, counting multiplicities. Since no point appears more than $k$ times among these blocks, there are at least $ks/k = s$ distinct points among them. Thus by Hall's Marriage Theorem \cite[Theorem 1]{HallP}, we may select at least one system of distinct representatives. Think of this system as removing a single element from each block of $\mathcal B$ and thereby placing us in precisely the same situation we started with, except with $k$ reduced by $1$. Iterating this process, we construct a collection of $k$ systems of distinct representatives that are mutually disjoint by construction, yielding (ii). Conversely, given (ii), suppose some point $x$ occurs in $k+1$ blocks. As each occurrence of $x$ must be selected by precisely one of the $k$ systems of distinct representatives, one such system selects at least two occurrences of $x$, a contradiction. Since no point can occur more than $k$ times, each must occur precisely $k$ times, yielding (i).

To see the equivalence of (ii) and (iii), note that each system of distinct representatives of $\mathcal B$ gives rise to a (latin) column of $L$, and vice versa. In particular, the symbol $g(B)$ occurs at the intersection of the column indexed by the system of distinct representatives $g$ and the row indexed by block $B$.
\end{proof}

\begin{lemma}[Hall \cite{Hall}]\label{Lm:HallCompletion} Given $0\le k\le n$, any $n\times k$ latin rectangle containing symbols from $\{1,\dots,n\}$ can be extended to a latin square of order $n$.
\end{lemma}

\begin{proof}[Proof of Theorem \ref{Th:Designs}]
We have already shown that $\mathcal D(Q,S)$ is a $1$-$(n-m,m,m)$ design. Since its points form the set $Q\setminus S$ and it has $n-m$ blocks by definition, it is symmetric.

Conversely, suppose that $\mathcal D=(\mathcal P,\mathcal B)$ is a symmetric $t$-$(v,k,\lambda)$ design, that is, $b=v$. We construct $Q$ of order $n$ and $S\le Q$ of order $m$ in three steps; first the latin rectangle $L_2$, then $L_1$, and finally $L_3$, referring to Figure \ref{Fg:Cayley}.

Set $m=k$, $n-m=v=b$. Every element of $\mathcal P$ appears in precisely $r=k$ blocks of $\mathcal B$. By Lemma \ref{Lm:Marriage}, the blocks of $\mathcal B$ give rise to an $(n-m)\times m$ latin rectangle $L_2$ on $n-m$ symbols which we identify with the elements of $Q\setminus S$. We can arrange additional $m$ symbols into any (normalized) latin square $L_1$ and declare it a multiplication table of $S$. Altogether, $L_1\cup L_2$ form an $n\times m$ latin rectangle on $n$ symbols. By Lemma \ref{Lm:HallCompletion}, $L_1\cup L_2$ can be completed to a latin square $L$ of order $n$ with some $n\times (n-m)$ latin rectangle $L_3$. Upon rearranging the rows and columns of $L$, if necessary, we obtain a multiplication table of $Q$.
\end{proof}

An obvious question is whether Theorem \ref{Th:Designs} is of any utility in the ongoing search for symmetric $t$-designs. The answer is probably negative, but we would like to say the following:

When $Q$ is a group, the design $\mathcal D(Q,S)$ is highly but trivially structured (with repeated blocks). On the other hand, if $Q$ is a random loop, it is to be expected that $\mathcal D(Q,S)$ is going to be merely a $1$-design, not a $t$-design with $t>1$. It might therefore seem that interesting designs $\mathcal D(Q,S)$ could be constructed in varieties of loops that have nice algebraic properties but not quite as strong as groups. Using the \texttt{LOOPS} package \cite{LOOPS}, we have conducted a heuristic search in the varieties of Moufang and Bol loops, but we did not find any $t$-designs with $t>1$. The difficulty becomes apparent upon a closer inspection of the proof of Theorem \ref{Th:Designs}. While the design itself must be cooked up carefully in the latin rectangle $L_2$, the entire loop $Q$ can be obtained essentially randomly by adjoining the latin rectangles $L_1$ and $L_3$. It would be interesting to see if $L_1\cup L_3$ can be obtained in a systematic (that is, algebraic) fashion depending on $L_2$, hence resulting in interesting algebraic properties of the loop $Q$. To illustrate this idea, we forced certain algebraic properties and used the finite model builder \texttt{Mace4} to obtain the nice loop $Q$ of \S \ref{sec:intro} with $\mathcal D(Q,S)$ corresponding to  the projective plane of order $2$.

We conclude this section with a few questions concerning the cardinality of $\il{Q,S}$. If $m$ divides $n$, then $\il{Q,S}$ can be made as small as possible (containing only $n/m$ sets of size $m$ and, if $m<n$, the empty set) by choosing $Q$ to be the cyclic group $C_n$. How small can $\il{Q,S}$ be when $m$ does not divide $n$? How big can $\il{Q,S}$ be?

The set $\il{Q,S}$ can contain at most $n-m+1$ sets of order $m$ and this will happen precisely when the $n-m$ cosets $\{xS\,:\,   x\in Q\setminus S\}$ are distinct, that is, when $\mathcal D(Q,S)$ is a \emph{simple} (no repeated blocks) design. This can be easily achieved by placing the symbols $0$, $\dots$, $n-m-1$ into $n-m$ rows (forming $L_2$) so that the $i$th row reads
\begin{displaymath}
    (i \mod (n-m),\ (i+1) \mod (n-m),\ \dots,\ (i+m-1)\mod (n-m) ).
\end{displaymath}
Note that if $3m \le n$, then this construction also maximizes the number of singletons in $\il{Q,S}$ since the intersection of rows $i$ and $i+m-1$ is $\{i+m-1\}$.
However, we know neither how to maximize the number of $k$-element subsets of $\il{Q,S}$ for a general $k$, $1<k<m$, nor how to maximize the cardinality of $\il{Q,S}$. We therefore ask:

\begin{problem}\label{Pr:MaxSize}
Suppose that $1\le m\le n$ are integers such that $m\le \lfloor n/2\rfloor$, and let $\mathcal B$ be the blocks of a $1$-$(n-m,m,m)$ design. Let $\overline{\mathcal B}$ be the closure of $\mathcal B$ under intersections. How should $\mathcal B$ be chosen to maximize the cardinality of $\overline{\mathcal B}$, the number of $k$-element subsets of $\overline{\mathcal B}$?
\end{problem}

The restriction $m\le \lfloor n/2\rfloor$ in Problem \ref{Pr:MaxSize} is necessary in the context of loops due to Corollary \ref{Cr:PossibleOrders}. On the set-theoretical level, it makes sense to propose Problem \ref{Pr:MaxSize} for any $1\le m\le n$ and without the assumption that every $1\le i\le n-m$ appears in precisely $m$ blocks of $\mathcal B$.

While in this paper we have restricted our attention to cosets of subloops, one could also consider whether interesting designs arise as translates of arbitrary subsets of $Q$. For example, a simple computer search (aided by the \texttt{DESIGN} package \cite{DESIGN} for \texttt{GAP}) reveals that in the Moufang loops of order $16$ with indices 2, 3, and 5 in the \texttt{LOOPS} package \cite{LOOPS}, there are, respectively, 128, 896, and 256 subsets of order $6$ whose collection of left translates form 2-(16,6,2) designs. In this case, these designs are all representatives of a single isomorphism class, and this class can in fact be realized by difference sets in groups of order $16$.

\section{Coset decompositions and Lagrange-like properties}\label{Sc:Decompositions}

Following \cite{Pf}, we say that $Q$ has a \emph{left coset decomposition modulo $S$} if any two left cosets of $S$ in $Q$ are either disjoint or coincide. As a weaker condition, we say that $Q$ has a \emph{left coset partition modulo $S$} if there is a subset of left cosets of $S$ in $Q$ that partitions $Q$.

\begin{lemma}[Theorem I.2.12 of \cite{Pf}]\label{Lm:Pf}
Let $S$ be a subloop of $Q$. Then $Q$ has a left coset decomposition modulo $S$ if and only if $(xs)S=xS$ for every $x\in Q$, $s\in S$.
\end{lemma}
\begin{proof}
Suppose that $Q$ has a left coset decomposition modulo $S$, and let $x\in Q$, $s\in S$. Since $xs\in (xs)S\cap xS$, we conclude that $(xs)S=xS$. Conversely, suppose that $(xs)S=xS$ for every $x\in Q$, $s\in S$. If $yS\cap zS\ne\emptyset$ then there are $s_1$, $s_2\in S$ such that $ys_1=zs_2$, so $yS = (ys_1)S = (zs_2)S = zS$.
\end{proof}

A loop has the \emph{right inverse property} if it satisfies the identities $yy^{-1}=y^{-1}y=1$ and $(xy)y^{-1}=x$. A loop is \emph{power-associative} if each element generates a group, and a power-associative loop is \emph{right power alternative} if $(xy^i)y^j = xy^{i+j}$ holds for all integers $i$, $j$. Note that a right power alternative loop has the right inverse property.

\begin{lemma}\label{Lm:RAlt}
Let $Q$ be a right power alternative loop and let $S\le Q$ be generated by one element (hence $S$ is a cyclic group). Then $Q$ has a left coset decomposition modulo $S$.
\end{lemma}
\begin{proof}
By Lemma \ref{Lm:Pf}, it suffices to show that $(xs)S = xS$ for every $x\in Q$ and $s\in S$. Since $S$ is a cyclic group, we can assume that $S = \langle t\rangle$, and we must prove that $(xt^n)S = xS$ for every $n$. This equality follows from $(xt^n)t^m = xt^{n+m}\in xS$ and $xt^m = (xt^n)(t^{m-n})\in (xt^n)S$.
\end{proof}

A loop is \emph{(right) Bol}, see \cite{Robinson}, if it satisfies the identity
\begin{equation}\label{Eq:RBol}
    x((yz)y) = ((xy)z)y.
\end{equation}
\emph{Left Bol} loops are defined by an identity dual to \eqref{Eq:RBol}.

A loop is Moufang if and only if it is both left and right Bol. Right Bol loops are right power alternative, and hence have the right inverse property. Consequently, by Lemma \ref{Lm:RAlt}, if $x$ is an element of a right Bol loop $Q$ then $Q$ has a left coset decomposition modulo $\langle x\rangle$; in particular, the order of $x$ divides the order of $Q$. This brings us to the other concept we wish to investigate in this section.

We say that a subloop $S$ of $Q$ is \emph{Lagrange-like} if $|S|$ divides $|Q|$. If $S\le Q$ and $Q$ is a group, then $S$ is Lagrange-like. We have just shown that if $S\le Q$, $S$ is cyclic and $Q$ is right power alternative then $S$ is Lagrange-like.

Using ideas similar to those of Glauberman \cite{Gl}, Foguel, Kinyon and Phillips proved in \cite{KiTu} that $S\le Q$ is Lagrange-like if $Q$ is a Bol loop of odd order. It is not known if Lagrange's Theorem holds for Bol loops. We present a novel technique by which it is possible to prove computationally that certain small subloops $S$ are Lagrange-like in any Bol loop $Q$ with $S\le Q$. In some instances we can show even more, namely that any (right) Bol loop $Q$ with $S\le Q$ has a (left) coset partition modulo such a subloop $S$.

For a loop $Q$ and $x\in Q$, let $R_x:Q\to Q$, $y\mapsto yx$ be the \emph{right translation by $x$}. Let $\rmlt(Q) = \langle R_x\,:\,   x\in Q\rangle$ be the permutation group generated by all right translations, the \emph{right multiplication group of $Q$}. For $S\le Q$, let $\rmlt(Q,S) = \langle R_x\,:\, x\in S\rangle$ be the \emph{relative right multiplication group of $Q$ with respect to $S$}. Both $\rmlt(Q)$ and $\rmlt(Q,S)$ act naturally on $Q$ and partition the elements of $Q$ into orbits. The orbit of $x\in Q$ under $\rmlt(Q,S)$ will be denoted by $O_x(Q,S)$. (Of course, the unique orbit of $\rmlt(Q)$ is all of $Q$.) We immediately have:

\begin{lemma} Let $S$ be a subloop of $Q$.
\begin{enumerate}
\item[(i)] If $|O_x(Q,S)|$ is a multiple of $|S|$ for every $x\in Q$, then $S$ is Lagrange-like in $Q$.
\item[(ii)] If $O_x(Q,S)$ can be written as a disjoint union of left cosets of $S$ for every $x\in Q$, then $Q$ has a left coset partition modulo $S$.
\end{enumerate}
\end{lemma}
\begin{proof}
Both claims follow immediately from the fact that the orbits partition $Q$.
\end{proof}

Note that a loop $Q$ has the right inverse property if and only if $R_x^{-1} = R_{x^{-1}}$ for all $x\in Q$. Therefore, in a right inverse property loop $Q$, we have
\begin{displaymath}
    O_x(Q,S) = \{R_{s_k}R_{s_{k-1}}\cdots R_{s_1}(x)\,:\,   k\ge 1,\,s_i\in S\text{ for }1\le i\le k\}.
\end{displaymath}

The difficulty we are facing is that we need to calculate $O_x(Q,S)$ for a fixed subloop $S$ of an unspecified right Bol loop $Q$. We therefore do not know the right translations $R_{s_i}$, but we can use the following greedy algorithm:

\vskip 3mm
\hrule
\vskip 3mm

\textbf{The algorithm}

\textit{Input}: A right Bol loop $S = \{s_1,\dots,s_m\}$ with neutral element $s_1$.

\textit{Output}: If the algorithm terminates, it returns all potential orbits $O_x(Q,S)$ in all Bol loops $Q$, possibly infinite, with $S\le Q$. More concretely, the algorithm returns a list of latin rectangles whose columns are indexed by $S$. Every possible action of $\rmlt(Q,S)$ on the orbit $O_x(Q,S)$ corresponds to one of these latin rectangles in the sense that the column $s_i$ lists the images of $R_{s_i}$ (the value $R_{s_i}(j)$ can be found in row $j$ and column $s_i$).

Note well that we do not assume that $Q$ is finite, nor do we claim that all returned potential orbits actually occur as $O_x(Q,S)$ in some $Q$, but we do claim that all actual orbits $O_x(Q,S)$ are on the list.

\textit{Initialization}: Without loss of generality, label the element $x$ as $1$, and label the distinct elements $1\cdot s_i$ as $i$, for $1\le i\le m$. Since the elements $1$, $\dots$, $m$ are in the orbit $O_x(Q,S)$ and since $s_1$ is the neutral element, we start the algorithm with this partially filled multiplication table
\begin{displaymath}
    \begin{array}{c|cccc}
        &s_1&s_2&\cdots&s_m\\
        \hline
        1&1&2&\cdots&m\\
        2&2& & &\\
        \vdots&\vdots& & &\\
        m&m& & &
    \end{array}
\end{displaymath}
in which the bottom right $(m-1)\times (m-1)$ subsquare is empty.

\textit{Recursive step}:

\textit{(a) Fill forced entries}: Locate the first pair of rows $a$, $b$ and columns $s_i$, $s_j$ such that the entries $(a,s_i)$, $(b,s_j)$ are already filled and $as_i = bs_j$. Then we must have $b = (as_i)s_j^{-1}$ thanks to the right inverse property, and $bs_i = ((as_i)s_j^{-1})s_i = a((s_is_j^{-1})s_i)$ by the right Bol identity \eqref{Eq:RBol}. Since $S$ is given, we can calculate $(s_is_j^{-1})s_i$, say it is equal to some $s_k\in S$. The entries $(b,s_i)$ and $(a,s_k)$ should therefore be the same and one of the following scenarios occurs.

If neither $(b,s_i)$ nor $(a,s_k)$ is known, then move on to the next pair $a$, $b$ of rows and $s_i$, $s_j$ of columns without filling any new entry. If precisely one of $(b,s_i)$, $(a,s_k)$ is known then fill the other entry; if the latin property has just been violated, then we have reached a contradiction and we backtrack, else we repeat step (a). Suppose that both $(b,s_i)$ and $(a,s_k)$ are already filled. If $bs_i\ne as_k$ we have reached a contradiction and we backtrack. If $bs_i=as_k$, we take no action and repeat (a).

If there are no suitable pairs of rows $a$, $b$ and columns $s_i$, $s_j$, move on to step (b).

\textit{(b) Fill greedily the first empty entry}: If the array is already completely filled, add it to the output list and backtrack. Else let $(a,s_i)$ be the first empty entry, and suppose that we have labeled $\ell$ rows so far. The suitable candidates $C$ for the value of $(a,s_i)$ consist of: (i) all symbols from $\{1,\dots,\ell\}$ not contained in the row $a$ or in the column $s_i$ and (ii) the new symbol $\ell+1$. Using depth first search, try each candidate in $C$ as the value for $(a,s_i)$. Should $\ell+1$ be used, also create a new row labeled with $\ell+1$ and fill the entry $(\ell+1,s_1)$ with $\ell+1$. Go to step (a).

\vskip 3mm
\hrule
\vskip 3mm

Let us illustrate the algorithm with two examples.

\begin{example}
Let $S=\{s_1,s_2,s_3\}\cong C_3$, where $s_1$ is the neutral element. We start with the partially filled array
\begin{displaymath}
    \begin{array}{c|ccc}
        &s_1&s_2&s_3\\
        \hline
        1&1&2&3\\
        2&2&&\\
        3&3&&
    \end{array}.
\end{displaymath}
Since $1\cdot s_2 = 2\cdot s_1$, we must have $2\cdot s_2 = ((1\cdot s_2)s_1^{-1})s_2 = 1((s_2s_1^{-1})s_2) = 1s_3 = 3$. This produces no conflict with the latin property, so we can force
\begin{displaymath}
    \begin{array}{c|ccc}
        &s_1&s_2&s_3\\
        \hline
        1&1&2&3\\
        2&2&3&\\
        3&3&&
    \end{array}
\end{displaymath}
and move on. It is an easy calculation to see that we are in fact forced to fill
\begin{displaymath}
    \begin{array}{c|ccc}
        &s_1&s_2&s_3\\
        \hline
        1&1&2&3\\
        2&2&3&1\\
        3&3&1&2
    \end{array}
\end{displaymath}
without ever having to resort to the greedy step (b) of the algorithm.

A similar argument works for every cyclic group $S=C_m$, offering an alternative proof of the special case of Lemma \ref{Lm:RAlt} concerned with right Bol loops $Q$.
\end{example}

\begin{example}
Let $S = \langle \sigma,\,\rho\,:\,   \sigma^2=\rho^3=(\sigma\rho)^2 =1\rangle \cong S_3$. Then one of the potential orbits returned by the algorithm is given in Figure \ref{Fg:S3}. Note that the first new symbol (row) added during the run of the algorithm is located in row $2$ and column $\sigma\rho$. Also note that $18$, the size of the potential orbit, is divisible by $6$, the order of $S_3$. Moreover, the potential orbit can be partitioned as a disjoint union of left cosets of $S_3$, for instance using the rows labeled $1$, $15$ and $17$.
\end{example}

\begin{figure}[ht]
\begin{displaymath}
\begin{array}{r|rrrrrr}
     &\mathrm{id}&\sigma&\sigma\rho&\rho&\rho^2&\sigma\rho^2\\
     \hline
     1&1& 2& 3& 4& 5& 6 \\
     2&2& 1& 7& 8& 9& 10 \\
     3&3& 10& 1& 11& 12& 7 \\
     4&4& 9& 12& 5& 1& 13 \\
     5&5& 8& 11& 1& 4& 14 \\
     6&6& 7& 10& 14& 13& 1 \\
     7&7& 6& 2& 15& 16& 3 \\
     8&8& 5& 16& 9& 2& 17 \\
     9&9& 4& 15& 2& 8& 18 \\
     10&10& 3& 6& 18& 17& 2 \\
     11&11& 17& 5& 12& 3& 16 \\
     12&12& 18& 4& 3& 11& 15 \\
     13&13& 15& 18& 6& 14& 4 \\
     14&14& 16& 17& 13& 6& 5 \\
     15&15& 13& 9& 16& 7& 12 \\
     16&16& 14& 8& 7& 15& 11 \\
     17&17& 11& 14& 10& 18& 8 \\
     18&18& 12& 13& 17& 10& 9
\end{array}
\end{displaymath}
\caption{A possible orbit $O_x(Q,S_3)$ in a right Bol loop $Q$.}\label{Fg:S3}
\end{figure}

We ran the algorithm for all right Bol loops of order less than $16$. The results are summarized in Figure \ref{Fg:Results}, which can be read as follows:

The first column gives the order $m$ of the subloop $S$. The second column gives $S$. Here we use standard notation when $S$ is a group (the group $G$ is the unique group of order $12$ not isomorphic to any of $C_3\times V_4$, $A_4$ or $D_{12}$), $M(S_3,2)$ is the unique nonassociative Moufang loop of order $12$ from \S \ref{sec:intro}, and $RightBol(m,i)$ denotes the $i$th nonassociative right Bol loop of order $m$, as cataloged by the \texttt{LOOPS} package. Figure \ref{Fg:Results} therefore accounts for all noncyclic right Bol loops of order less than $16$. We omit $S\cong C_m$ from the figure since that case is covered by Lemma \ref{Lm:RAlt}.

In the third column, we list sizes of all potential orbits $O_x(Q,S)$ returned by the algorithm, and in parentheses we offer the number of times a given length has been returned (for purposes of independent verification of our data). These multiplicities are also of interest since two potential orbits of the same size may have significantly different internal structures. The last column says ``yes'' when every potential orbit returned by the algorithm can be written as a disjoint union of some left cosets of $S$. When the last column says ``?'', at least one potential orbit could not be so decomposed. Notice, however, that this does not necessarily mean that $Q$ does not have a left coset partition modulo $S$ because we do not know which potential orbits returned by the algorithm occur as actual orbits.

\begin{figure}
\begin{displaymath}
\begin{array}{rrrc}
    \text{order}&\text{subloop }S&\text{orbit lengths (occur $\times$ times)}&\text{partition mod $S$?}\\
    \hline\hline
    4&V_4&4,\,8&\text{yes}\\
    \hline
    6&S_3&6(2),\,18&\text{yes}\\
    \hline
    8&C_2\times C_4&8(2),\,16&\text{yes}\\
     &C_2\times C_2\times C_2&8(30),\,16(1605),\,32(1225),\,64(99),\,128&\text{?}\\
     &D_8&8(2),\,16,\,32&\text{yes}\\
     &Q_8&8(2),\,16&\text{yes}\\
     &RightBol(8,1)&8(2),\,16(7),\,32&\text{?}\\
     &RightBol(8,2)&8(2),\,16(7),\,32&\text{?}\\
     &RightBol(8,3)&8(2),\,16(7),\,32&\text{?}\\
     &RightBol(8,4)&8(2),\,16&\text{yes}\\
     &RightBol(8,5)&8(2),\,16&\text{yes}\\
     &RightBol(8,6)&8(2),\,16&\text{yes}\\
     \hline
     9&C_3\times C_3&9,\,27&\text{yes}\\
     \hline
    10&D_{10}&10(4),\,50&\text{yes}\\
    \hline
    12&C_3\times V_4&12,\,24&\text{yes}\\
      &A_4&12(2),\,24(6),\,48,\,96&\text{yes}\\
      &D_{12}&12(2),\,24(2),\,36,\,72&\text{yes}\\
      &G&12(2),\,36&\text{yes}\\
      &M(S_3,2)&\,12(24),\,24(8),\,36(756),\,72(84)&\text{?}\\
      & & \,108(972),\,216(36),\,324(81),\,648\\
      &RightBol(12,1)&12(6),\,24(2),\,36(9),\,72&\text{yes}\\
      &RightBol(12,2)&12(6),\,24(2),\,36(9),\,72&\text{yes}\\
    \hline
    14&D_{14}&14(6),\,98&\text{yes}\\
    \hline
    15&RightBol(15,1)&15(3),\,75&\text{yes}\\
      &RightBol(15,2)&15(3),\,75&\text{yes}\\
    \hline
\end{array}
\end{displaymath}
\caption{Lengths of potential orbits $O_x(Q,S)$ in a right Bol loop $Q$ for small subloops $S\le Q$.}\label{Fg:Results}
\end{figure}

In summary:

\begin{theorem}\label{Th:RBol}
Let $Q$ be a right Bol loop, possibly infinite.
\begin{enumerate}
\item[(i)] Let $S$ be a right Bol loop of order less than $16$ or a finite cyclic group, and suppose that $S\le Q$. Then the length of every orbit $O_x(Q,S) = \rmlt(Q,S)(x)$ is divisible by $|S|$. In particular, if $Q$ is finite, then $|S|$ divides $|Q|$.
\item[(ii)] Let $S$ be a right Bol loop of order less than $16$ or a finite cyclic group, except for $C_2\times C_2\times C_2$, $RightBol(8,1)$, $RightBol(8,2)$, $RightBol(8,3)$, $M(S_3,2)$. If $S\le Q$ then $Q$ has a left coset partition modulo $S$.
\end{enumerate}
\end{theorem}
\begin{proof}
Independent implementations of the above algorithm were written by two of the authors. Their results agreed and are presented in Figure \ref{Fg:Results}, from which the result follows.
\end{proof}

To our knowledge, the results of Theorem \ref{Th:RBol} are new whenever $S$ is not a cyclic group. With regard to Lagrange's Theorem (that is, $|S|$ divides $|Q|$ in part (i)), the results are new whenever $|S|$ cannot be expressed as the least common multiple of orders of certain elements of $S$, for instance when $S=V_4$ or $S=A_4$.

Note that we do not know if a left coset partition modulo $S$ exists in the exceptional cases of Theorem \ref{Th:RBol}(ii). Also note the rather astonishing lengths of some potential orbits returned by the algorithm, say the one of length $648$ obtained with $S=M(S_3,2)$. We therefore ask:

\begin{problem}\label{Pr:Algorithm}
Let $S$ be a fixed right Bol loop and $Q$ an unspecified right Bol loop with $S\le Q$.
\begin{enumerate}
\item[(i)] Will the algorithm always terminate? (The algorithm could fail to terminate for at least two reasons: some potential orbit is infinite or there are finite potential orbits of arbitrarily large size. We were not able to rule out either of these possibilities.)
\item[(ii)] Are all potential orbits returned by the algorithm also actual orbits?
\item[(iii)] Is there an upper bound in terms of $m=|S|$ on the size of potential orbits returned by the algorithm?
\item[(iv)] Is there an upper bound in terms of $m=|S|$ on the size of the actual orbits $O_x(Q,S)$?
\item[(v)] Is $|O_x(Q,S)|$ always divisible by $|S|$?
\item[(vi)] Does every $O_x(Q,S)$ decompose as a disjoint union of some left cosets of $S$?
\end{enumerate}
\end{problem}

Answering (v) affirmatively would imply Lagrange's Theorem for Bol loops. Answering (vi) affirmatively would imply that a right Bol loop $Q$ has a left coset partition modulo $S$ whenever $S\le Q$, a result at least as strong as Lagrange's Theorem.

\section{Intersections of cosets}

Problem \ref{Pr:Algorithm} is open even for Moufang loops. We restate this special case of Problem \ref{Pr:Algorithm}(vi) here to obtain the following long-standing open problem:

\begin{problem}\label{Pr:MoufDecomp}
Let $Q$ be a Moufang loop and $S\le Q$. Does $Q$ have a left coset partition modulo~$S$?
\end{problem}

With regards to Problem \ref{Pr:MoufDecomp}, it is known that $Q$ need not have a left coset decomposition modulo $S$, i.e., distinct Moufang cosets can have non-trivial intersections despite the fact that the order of each coset must divide $|Q|$. To approach this problem, therefore, we are interested in properties of nonempty coset intersections $xS\cap yS$.

Suppose that $Q$ is a right Bol loop, $S\le Q$, and $x, y\in Q$ are such that $xS\cap yS\ne\emptyset$. Define
\begin{displaymath}
    f_{x,y}:xS\cap yS\to xS\cap yS,\quad xs\mapsto ys.
\end{displaymath}
This indeed defines a mapping, since if $xs\in xS\cap yS$ then $xs=yr$ for some $r\in S$, and we have $y = (xs)r^{-1}$, $ys = ((xs)r^{-1})s = x((sr^{-1})s) \in xS$ by the right Bol identity \eqref{Eq:RBol}. Since $f_{x,y}$ is clearly one-to-one, it is a permutation of $xS\cap yS$.

Recall that in a Bol loop $Q$ the orders of elements divide the order of $Q$ (cf. Lemma \ref{Lm:RAlt}). The following result hence poses a mild restriction on the possible sizes of $xS\cap yS$.

\begin{lemma}\label{Lm:SumOrders}
Let $Q$ be a Moufang loop, $S\le Q$, and $x, y\in Q$ such that $xS\cap yS\ne\emptyset$. Let $xs=yr\in xS\cap yS$. Then $xs$ belongs to a cycle of $f_{x,y}$ whose length is $|sr^{-1}|$, $sr^{-1}\in S$.

In particular, when $S \neq 1$, $|xS \cap yS|$ can be written as a sum of orders of some (possibly repeated) nonidentity elements of $S$ (hence of $Q$).
\end{lemma}
\begin{proof}
Let $f=f_{x,y}$ and suppose that $xs=yr$ for some $s$, $r\in S$. We claim that $f^k(xs) = x\cdot t_k$ for every $k\ge 0$, where $t_k = (sr^{-1})^ks\in S$. Note that by diassociativity $t_k t_{k-1}^{-1} t_k$ is well defined and equal to $t_{k+1}$.

The claim is certainly true for $k=0$. We have $f(xs) = ys = (xs\cdot r^{-1}) s = x\cdot (sr^{-1})s$, so the claim is true for $k=1$. Suppose that the claim is true for $k$ and $k-1$. Then $f^{k+1}(xs) = f(f^k(xs)) = f(x\cdot t_k)$ and also $f^{k+1}(xs) = f(f(f^{k-1}(xs))) = f(f(x\cdot t_{k-1})) = f(y\cdot t_{k-1})$. This means that $x\cdot t_k = y\cdot t_{k-1}$, or $y = (xt_k)t_{k-1}^{-1}$. Thus $f^{k+1}(xs) = f(x\cdot t_k) = yt_k = (xt_k)t_{k-1}^{-1}\cdot t_k = x(t_kt_{k-1}^{-1}t_k) = xt_{k+1}$. This completes the proof of the claim.

Since $f^k(xs) = x\cdot (sr^{-1})^ks$ is equal to $xs$ if and only if $(sr^{-1})^k=1$, if follows that $xs$ is in a cycle of $f$ of length $|sr^{-1}|$. We conclude that all cycle lengths of $f$ have sizes corresponding to orders of elements of $S$. If none of these orders is $1$, then the claim follows. Note that if $|sr^{-1}| = 1$, then $s=r$, $x = y$, and $xS = yS$. The claim then holds in this case too, since $S$ must contain a non-identity element and by Lemma \ref{Lm:RAlt} its order will divide $|xS|=|S|$.
\end{proof}

%It follows from the second part of Lemma \ref{Lm:SumOrders} and from Lemma \ref{Lm:RAlt} that if $S$ is a subloop of even order in a Moufang loop $Q$ then $|xS \cap yS|$ must be even as well.

We will now obtain a stronger restriction on the cardinality of $xS\cap yS$ by attempting to shift (by a left translation) the set $xS\cap yS$ into a subloop.

Note first that we cannot necessarily assume without loss of generality (by suitably choosing the representatives of the two cosets) that $x\in xS\cap yS$, as the following example shows:

\begin{example}\label{Ex:Bad2Cosets}
Consider a loop $Q$ with elements $1\le i\le 7$ and subloop $S=\{1,2,3\}$ in which the latin rectangle $L_2$ of Figure \ref{Fg:Cayley} corresponding to rows $4$--$7$ and columns $1$--$3$ is filled as follows:
\begin{displaymath}
    \begin{array}{ccc}
        4&5&6\\
        5&4&7\\
        6&7&4\\
        7&6&5
    \end{array}
\end{displaymath}
Then the representatives of the left cosets $4S$ and $7S$ are uniquely determined (since all cosets $xS$ for $x\not\in S$ are distinct). Moreover, we have $4S\cap 7S = \{5,6\}\ne\emptyset$ but $4$, $7\not\in 4S\cap 7S$.
\end{example}

Dual to the right translations in a loop $Q$, we also define the \emph{left translations} $L_x : Q\to Q$, $y\mapsto xy$. These generate the \emph{left multiplication group} $\lmlt(Q) = \langle L_x\,:\, x\in Q\rangle$. The subgroup of $\lmlt(Q)$ stabilizing the identity element $1\in Q$ is called the \emph{left inner mapping group} $\linn(Q) = (\lmlt(Q))_1$.

\begin{lemma}\label{Lm:Replace}
Let $Q$ be a loop, let $S\leq Q$ and suppose $g\in \lmlt(Q)$. For each $x\in g(S)$, there exists $h_x\in \linn(Q)$ such that $g(S) = xh_x(S)$. In particular, if $1\in g(S)$, then there exists $h\in \linn(Q)$ such that $g(S) = h(S)$.
\end{lemma}
\begin{proof}
If $x\in g(S)$, then there exists $s\in S$ such that $g(s) = x$. Set $h_x = L_x^{-1} g L_s$. Then $h_x(1) = L_x^{-1}g(s) = L_x^{-1}(x) = 1$, so $h_x\in \linn(Q)$. Also, $xh_x(S) = g L_s(S) = g(S)$, as claimed.
\end{proof}

A loop $Q$ is said to be \emph{left automorphic} if every left inner mapping is an automorphism, that is, $\linn(Q) \leq \aut(Q)$. See \cite{BrPa} and \cite{JeKiVo} for an introduction to automorphic loops.

As is the custom in loop theory, we use $\backslash$ and $/$ to denote left and right division, respectively. That is, $x\backslash y= L_x^{-1}(y)$ and $x/y = R_y^{-1}(x)$. By \cite[Theorem I.2.3]{Pf}, a nonempty subset $S$ of a loop $Q$ is a subloop of $Q$ if and only if it is closed under multiplication and the left and right divisions.

\begin{lemma}\label{Lm:LAut}
Let $Q$ be a left automorphic loop, let $S\leq Q$ and let $g_i\in \lmlt(Q)$, $i\in I$. Set $H = \bigcap_{i\in I} g_i(S)$.
Then $x\in H$ if and only if $x\backslash H$ is a subloop of $Q$.
\end{lemma}
\begin{proof}
Assume $x\in H$. By Lemma \ref{Lm:Replace}, for each $i\in I$ there exists $h_i\in \linn(Q)$ such that $g_i(S) = xh_i(S)$, and so $H = \bigcap_{i\in I} xh_i(S) = x\bigcap_{i\in I}h_i(S)$. Let $y$, $z\in x\backslash H = \bigcap_{i\in I}h_i(S)$. Then for every $i\in I$ there are $s_i$, $s_i'\in S$ such that $y = h_i(s_i)$, $z=h_i(s_i')$, $yz = h_i(s_i)h_i(s_i') = h_i(s_is_i')\in h_i(S)$, $y\backslash z = h_i(s_i)\backslash h_i(s_i') = h_i(s_i\backslash s_i')\in h_i(S)$, and $y/z = h_i(s_i)/h_i(s_i') = h_i(s_i/s_i')\in h_i(S)$. Thus $yz$, $y\backslash z$, $y/z\in \bigcap_{i\in I}h_i(S) = x\backslash H$, and $x\backslash H\le Q$. For the converse, if $x\backslash H$ is a subloop, then $1\in x\backslash H$ and so $x\in x(x\backslash H) = H$.
\end{proof}

\begin{corollary}\label{Co:LAut}
Let $Q$ be a left automorphic loop, let $S\le Q$, and let $x, y\in Q$. If $x\in xS\cap yS$ then $x\backslash (xS\cap yS)$ is a subloop of $S$.\qed
\end{corollary}

Note that the class of left automorphic loops includes commutative Moufang loops by \cite[Lemma VII.2.2]{Br}, and \emph{conjugacy closed loops} (loops in which every $L_x^{-1}L_yL_x$ is a left translation and every $R_x^{-1}R_yR_x$ is a right translation) by \cite[Theorem 2.2]{GoRo}.

We observe that the conclusion of Corollary \ref{Co:LAut} cannot be strengthened to $x\backslash(xS\cap yS) = S$, as there is a commutative Moufang loop $Q$ of order $81$ with a subloop $S$ of order $9$ such that $|xS\cap yS|=3$ for a suitable choice of $x$ and $y$.

It is perhaps worth noting in passing (see below) that the asymmetry in the assumption $x\in xS\cap yS$ is illusory in Moufang loops. In Moufang loops, we can write $x^{-1}H$ instead of $x\backslash H$ thanks to the inverse property.

\begin{lemma}
Let $Q$ be a Moufang loop, let $S\leq Q$ and let $x,y\in Q$. Then $x^{-1}(xS\cap yS) = y^{-1}(xS\cap yS)$.
\end{lemma}
\begin{proof}
Let $D_x = x^{-1}(xS\cap yS)$ and $D_y =y^{-1}(xS\cap yS)$. Let $s\in D_x$. Then there is $t\in S$ such that $xs=yt$. To show that $s\in D_y$, it suffices to prove that $ys\in xS\cap yS$, that is, that $ys\in xS$. From $xs=yt$ we have $y=xs\cdot t^{-1}$ by the right inverse property, so $ys = (xs\cdot t^{-1})s = x(st^{-1}s)\in xS$. The other inclusion $D_y\subseteq D_x$ follows by symmetry.
\end{proof}

Finally, we obtain a restriction on the cardinality of $xS\cap yS$ in left automorphic Moufang loops. Note that in addition to commutative Moufang loops, the class of left automorphic Moufang loops also contains the \emph{extra loops} (defined by $x(y(zx)) = ((xy)z)x$ in \cite{Fe}).

\begin{theorem}
Let $Q$ be a left automorphic Moufang loop, $S\le Q$, and let $x$, $y\in Q$ be such that $xS\cap yS\ne\emptyset$. Then $|xS\cap yS|=|T|$ for some subloop $T$ of $S$. In particular, $|xS\cap yS|$ divides $|S|$.
\end{theorem}
\begin{proof}
By one of the Moufang identities \eqref{Eq:Moufang}, $s(aS)s = (sa)(Ss) = (sa)S$ for every $s\in S$, $a\in Q$. Since $xS\cap yS\ne\emptyset$, there is $s\in S$ such that $xs\in xS\cap yS$. Then $s^{-1}x = s^{-1}(xs)s^{-1}\in s^{-1}(xS\cap yS)s^{-1} = s^{-1}(xS)s^{-1} \cap s^{-1}(yS)s^{-1} = (s^{-1}x)S\cap (s^{-1}y)S$. As $|xS\cap yS| = |s^{-1}(xS\cap yS)s^{-1}| = |(s^{-1}x)S\cap (s^{-1}y)S|$, we can assume without loss of generality that $x\in xS\cap yS$. We are done by Corollary \ref{Co:LAut}.
\end{proof}

\begin{problem}
Let $Q$ be a loop, $S\le Q$, and let $x$, $y\in S$ be such that $xS\cap yS\ne\emptyset$.
\begin{enumerate}
\item[(i)] If $Q$ is Moufang, is $|xS\cap yS|=|T|$ for some $T\le Q$, some $T\le S$?
\item[(ii)] If $Q$ is left automorphic, is $|xS\cap yS|=|T|$ for some $T\le Q$, some $T\le S$?
\item[(iii)] If $Q$ is Moufang and $x\in xS\cap yS$, is $x^{-1}(xS\cap yS)$ a subloop of $S$?
\end{enumerate}
\end{problem}

\section*{Acknowledgement}

We thank the anonymous referees for several useful suggestions concerning the presentation of background material, especially on symmetric designs.

\end{document}